\documentclass[11pt]{amsart}

\usepackage[draft]{changes}

\definechangesauthor[color=red]{pe}
\definechangesauthor[color=red]{xm}
\definechangesauthor[color=red]{hx}
\newcommand{\stkout}[1]{\ifmmode\text{\sout{\ensuremath{#1}}}\else\sout{#1}\fi}
\setdeletedmarkup{\stkout{#1}}

\usepackage{amssymb}

\usepackage{graphics}
\usepackage{graphicx}

\usepackage{amsmath}
\usepackage{amsthm}
\usepackage{amsfonts}
\usepackage{mathtools}
\usepackage{xcolor}
\usepackage[english]{babel}
\usepackage[margin=1.0in]{geometry}
\usepackage[colorlinks=true,
        linkcolor=blue]{hyperref}
\usepackage{bbm}
\usepackage{verbatim}
\usepackage{extarrows}
\usepackage{blkarray}

\usepackage[T1]{fontenc}

\numberwithin{equation}{section}

\newtheorem{prop}{Proposition}
\newtheorem{lemma}[prop]{Lemma}

\newtheorem{thm}[prop]{Theorem}

\numberwithin{prop}{section}

\newtheorem{defn}[prop]{Definition}
\theoremstyle{definition}

\newtheorem{rmk}[prop]{Remark}

\definecolor{c1}{rgb}{0.2,0.4,0.5}
\definecolor{c2}{rgb}{0.1,0.3,0.5}
\definecolor{c3}{rgb}{0.2,0.7,0.5}
\usepackage{tikz}

\newcommand{\oo}[1]{\overline{#1}}

\newcommand{\bC}{\mathbb{C}}
\newcommand{\bR}{\mathbb{R}}
\newcommand{\bB}{\mathbb{B}}

\newcommand{\dbar}{\oo\partial}

\newcommand{\id}{\mathrm{id}}

\DeclareMathOperator{\Aut}{Aut}

\begin{document}

\title[]{On the classification of normal Stein spaces and finite ball quotients with Bergman-Einstein metrics}

\begin{abstract}
In this paper, we study the Bergman metric of a finite ball quotient $\mathbb{B}^n/\Gamma$, where $\Gamma\subseteq \Aut(\bB^n)$ is a finite, fixed point free, abelian group. We prove that this metric is K\"ahler--Einstein if and only if $\Gamma$ is trivial, i.e., when the ball quotient $\mathbb{B}^n/\Gamma$ is the unit ball $\bB^n$ itself. As a consequence, we establish a characterization of the unit ball among normal Stein spaces with isolated singularities and abelian fundamental groups in terms of the existence of a Bergman-Einstein metric.
\end{abstract}



\author [Ebenfelt]{Peter Ebenfelt}
\address{Department of Mathematics, University of California at San Diego, La Jolla, CA 92093, USA} \email{{pebenfelt@ucsd.edu}}

\author[Xiao]{Ming Xiao}
\address{Department of Mathematics, University of California at San Diego, La Jolla, CA 92093, USA}
\email{{m3xiao@ucsd.edu}}

\author [Xu]{Hang Xu}
\address{Department of Mathematics, University of California at San Diego, La Jolla, CA 92093, USA}
\email{{h9xu@ucsd.edu}}

\thanks{The first and second authors were supported in part by the NSF grants DMS-1900955 and DMS-1800549, respectively.}

\maketitle

\section{Introduction}

Since the introduction of the Bergman kernel in \cite{Bergman1933, Bergman1935} and the subsequent
groundbreaking work of Kobayashi \cite{Ko} and Fefferman \cite{Fe}, the study of the Bergman kernel and metric has been a central subject in several complex variables and complex geometry. A general problem of fundamental importance seeks to characterize complex analytic spaces in terms of geometric properties of their Bergman metrics. The Bergman kernel of the unit ball $\bB^n\subseteq \bC^n$, for example, is explicitly known,
\begin{equation*}
K_{\bB^n}(z,\bar w)=\frac{n!}{\pi^n}\frac{1}{(1-\left<z,\bar w\right>)^{n+1}},\quad \left<z,\bar w\right>=\sum_{j=1}^n z_j\bar w_j,
\end{equation*}
and it is routine to verify that the Bergman metric,
\begin{equation*}
(g_{\bB^n})_{i\bar j}=\frac{\partial^2}{\partial{z_i}\partial{\bar z_j}}\log K_{\bB^n}(z,\bar z),
\end{equation*}
is {\it K\"ahler-Einstein}, i.e., has Ricci curvature equal to a constant multiple of the metric tensor; indeed, the Bergman metric of the unit ball has constant holomorphic sectional curvature, which implies the K\"ahler-Einstein property.
A well-known conjecture posed by S.-Y. Cheng \cite{Ch} in 1979 asserts that the Bergman metric of a bounded, strongly pseudoconvex domain in $\bC^n$ with smooth boundary is K\"ahler-Einstein if and only if the domain is biholomorphic to the unit ball $\mathbb{B}^n$. There are also variations of this conjecture in terms of other canonical metrics; see, e.g., Li \cite{Li05, Li09, Li16} and references therein.

The aforementioned Cheng Conjecture was confirmed by S. Fu--B. Wong \cite{FuWo} and S. Nemirovski--R. Shafikov \cite{NeSh} in the two dimensional case and by X. Huang and the second author \cite{HX} in higher dimensions. X. Huang and X. Li \cite{HuLi} recently generalized this result to Stein manifolds with strongly pseudoconvex boundary as follows: {\em The only Stein manifold with smooth and compact strongly pseudoconvex boundary for which the Bergman metric is K\"ahler-Einstein is the unit ball $\bB^n$ (up to biholomorphism)}.
These results lead naturally to the question of whether a similar characterization of $\bB^n$ holds in the setting of normal Stein spaces {\em with possible singularities}; see Conjecture 1.4 in \cite{HX20}. In this paper, we provide strong evidence that this is the case. The following two theorems establish the first results that the authors are aware of characterizing the unit ball among normal Stein spaces with possible singularities in terms of the existence of a Bergman-Einstein metric.

\begin{thm}\label{Thm 1}
	Let $V$ be an $n$-dimensional Stein space in $\bC^N$ with $N>n\geq 2$, and $G=V\cap \mathbb{B}^N$. Assume that every point in $\oo{G}$ is a smooth point of $V$, except for finitely many normal singularities in $G$, and that $G$ has a smooth boundary.  Then the Bergman metric of $G$ is K\"ahler-Einstein if and only if $G$ is biholomorphic to $\mathbb{B}^n$.
\end{thm}

\begin{thm}\label{Thm 2}
	Let $V$ be an $n$-dimensional Stein space in $\bC^N$ with $N>n\geq 2$ and $\Omega \subseteq \bC^N$ a bounded strongly pseudoconvex domain with smooth and real-algebraic boundary. Write $G=V \cap \Omega.$ Assume every point in $\oo{G}$ is a smooth point of $V$, except for finitely many normal singularities in $G$, and that $G$ has a smooth boundary. Then the following are equivalent:
	\begin{itemize}
		\item[(i)] $G$ is biholomorphic to $\mathbb{B}^n$.
		\item[(ii)] The fundamental group of the regular part of $G$ is abelian and the Bergman metric of $G$ is K\"ahler-Einstein.
	\end{itemize}
\end{thm}

\begin{rmk}
As we will see in the proof (Section 3), if $G$ itself is assumed to be bounded in Theorem \ref{Thm 2}, then the boundedness assumption on $\Omega$ can be dropped.
\end{rmk}

 We shall utilize the work of D'Angelo--Lichtblau \cite{DALi} (see also F. Forstneri\v{c} \cite{Fo86}) and X. Huang \cite{Hu}, as well as methods from \cite{HX}, \cite{HuLi} and \cite{EbenfeltXiaoXu2020algebraicity} to reduce the proofs of Theorems \ref{Thm 1} and \ref{Thm 2} to that of the following theorem, which is one of the main results in the paper.

\begin{thm}\label{main thm}
	Let $\Gamma$ be a finite abelian subgroup of $\Aut(\mathbb{B}^n)$, $n\geq 2$, and assume $\Gamma$ is fixed point free. Then the Bergman metric of $\mathbb{B}^n/\Gamma$ is K\"ahler-Einstein if and only if $\Gamma$ is the trivial group consisting of the identity element.
\end{thm}

Here a subgroup $\Gamma$ of $\Aut(\mathbb{B}^n)$ is called \emph{fixed point free} if the only element $\gamma\in\Gamma$ with a fixed point on $\partial\mathbb{B}^n$ is the identity.
The fixed point free condition on $\Gamma$ guarantees that the quotient space $\mathbb{B}^n/\Gamma$ has smooth boundary (see \cite{Fo86}). Moreover, as we shall see in Section \ref{Sec main thm proof}, an abelian fixed point free finite group $\Gamma$
is in fact cyclic.

To prove Theorem \ref{main thm}, it suffices to show that if $\Gamma$ is not the trivial group, i.e., $\Gamma\neq \{\id\}$, then the Bergman metric is not K\"ahler-Einstein. For that, we shall use the transformation formula for the Bergman kernel under branched holomorphic coverings of complex analytic spaces; see Theorem \ref{BK transformation thm} below. A crucial step in the proof is to reduce the non-Einstein condition to several combinatorial inequalities. The proofs of these combinatorial inequalities are technical and will be given in a separate section; see Section \ref{Sec combinatorial lemma}.

We remark that the analogue of Theorem \ref{main thm} is not true in the case $n=1$. If we denote the unit disk in $\bC$ by
$\mathbb{D}$ ($=\bB^1$), then one readily verifies that any finite subgroup $\Gamma \subseteq \mathrm{Aut}(\mathbb{D})$ must be fixed point free and cyclic. Nevertheless, in this case,
X. Huang and X. Li \cite{HuLi} proved the very interesting result that the Bergman metric of $\mathbb{D}/\Gamma$ always has constant Gaussian curvature, which is equivalent to being K\"ahler-Einstein in the one-dimensional case.


The paper is organized as follows. Section \ref{Sec background} recalls some preliminaries on the Bergman metric and finite ball quotients. In Section \ref{Sec Thm 1 and 2}, we prove that Theorems \ref{Thm 1} and \ref{Thm 2} follow from Theorem \ref{main thm}. Theorem \ref{main thm} is then proved in Section \ref{Sec main thm proof}, except for some combinatorial lemmas used in the proof that are left to Section \ref{Sec combinatorial lemma}.


\section{Preliminaries}\label{Sec background}

\subsection{The Bergman kernel}
In this subsection, we will briefly review some properties of the Bergman kernel and metric on a complex manifold. More details can be found in \cite{Ko} and \cite{KoNo}.

Let $M$ be an $n$-dimensional complex manifold. Let $L^2_{(n,0)}(M)$ denote the space of $L^2$-integrable $(n,0)$-forms on $M,$ equipped with the inner product
\begin{equation}\label{inner product}
(\varphi,\psi)_{L^2(M)}:=i^{n^2}\int_{M}\varphi\wedge\oo{\psi},
\quad \varphi,\psi \in L^2_{(n,0)}(M).
\end{equation}

Define the {\em Bergman space} of $M$  to be
\begin{equation}\label{Bergman space of forms}
A^2_{(n,0)}(M):=\bigl\{\varphi \in L^2_{(n,0)}(M): \varphi \mbox{ is a holomorphic $(n,0)$-form on $M$} \}.
\end{equation}

Assume $A^2_{(n,0)}(M) \neq \{0\}$. Then $A^2_{(n,0)}(M)$ is a separable Hilbert space. Taking any orthonormal basis $\{\varphi_k\}_{k=1}^{q}$ of $A^2_{(n,0)}(M)$ with $1 \leq q \leq \infty$, we define the {\em Bergman kernel (form)} of $M$ to be
\begin{equation*}
K_M(x,\bar{y})=i^{n^2}\sum_{k=1}^{q} \varphi_k(x)\wedge \oo{\varphi_k(y)}.
\end{equation*}
Then, $K_M(x, \bar{x})$ is a real-valued, real-analytic form of degree $(n,n)$ on $M$ and is independent of the choice of orthonormal basis.

The Bergman kernel form remains unchanged if we remove a proper complex analytic subvariety, as the following theorem from \cite{Ko} shows:
\begin{thm}[\cite{Ko}]\label{Kobayashi thm}
	If $M'$ is a domain in an $n$-dimensional complex manifold $M$ and if $M-M'$ is a complex analytic subvariety of $M$ of complex dimension $\leq n-1$, then
	\begin{equation*}
		K_M(x,\bar{y})=K_{M'}(x,\bar{y}) \quad \mbox{ for any } x, y \in M'.
	\end{equation*}
\end{thm}

This theorem suggests the following generalization of the Bergman kernel form to complex analytic spaces.
\begin{defn}\label{defnsp}
Let $M$ be a reduced complex analytic space, and let $V\subseteq M$ denote its set of singular points. The Bergman kernel form of $M$ is defined as
	\begin{equation*}
		K_M(x,\bar{y})=K_{M-V}(x,\bar{y})\quad \mbox{ for any } x, y\in M-V,
	\end{equation*}
	where $K_{M-V}$ denotes the Bergman kernel form of the complex manifold consisting of regular points of $M$.
\end{defn}

Let $N_1, N_2$ be two complex manifolds of dimension $n$. Let $\gamma: N_1\rightarrow M$ and $\tau: N_2\rightarrow M$ be holomorphic maps. The pullback of the Bergman kernel $K_M(x,\bar{y})$ of $M$ to $N_1 \times N_2$ is defined in the standard way. That is, for any $z\in N_1, w\in N_2$,
\begin{equation*}
	\bigl( (\gamma,\tau)^*K_M\bigr)(z,\bar{w})=\sum_{k=1}^{q} \gamma^*\varphi_k(z)\wedge \oo{\tau^*\varphi_k(w)}.
\end{equation*}
In terms of local coordinates, we may write the Bergman kernel form of $M$ as
\begin{equation}\label{BK local coordinates 1}
	K_M(x,\bar{y})=\widetilde{K}_M(x,\bar{y})dx_1\wedge\cdots dx_n\wedge d\oo{y_1}\wedge\cdots \wedge d\oo{y_n},
\end{equation}
where the function $\widetilde{K}_M(x,\bar{y})$ depends on the choice of local coordinates.
We then have
\begin{equation}\label{BK local coordinates 2}
	\bigl( (\gamma,\tau)^*K_M\bigr)(z,\bar{w})=\widetilde{K}_M(\gamma(z),\oo{\tau(w)})\,J_{\gamma}(z)\,\oo{J_{\tau}(w)}\,dz_1\wedge\cdots dz_n\wedge d\oo{w_1}\wedge\cdots \wedge d\oo{w_n},
\end{equation}
where $J_{\gamma}$ and $J_{\tau}$ are the Jacobian determinants of the maps $\gamma$ and $\tau$, respectively. In particular, we observe that the kernel function $\widetilde{K}_M(x,\bar{y})$ transforms accordining to the usual biholomorphic invariance formula under changes of local coordinates.

Let $M$ be as in Definition \ref{defnsp}. Assume $K_M(x,\bar{x})$ is non-vanishing (on the set of regular points of $M$, where it is defined). We define a Hermitian $(1,1)$-form on the regular part of $M$ by
\begin{equation}\label{Bergman metric def}
	\omega_M:=i \, \partial\dbar \log \widetilde{K}_M(x,\bar{x}).
\end{equation}
The biholomorphic invariance of the Bergman kernel implies that this form is independent of the choice of local coordinates used to determine the function $\widetilde{K}_M(x,\bar{x})$.
The Bergman metric on $M$ is the metric induced by $\omega_M$ (when it indeed induces a positive definite metric on the regular part of $M$).

We recall the Bergman kernel transformation formula in \cite{EbenfeltXiaoXu2020algebraicity} for (possibly branched) covering maps of complex analytic spaces. This formula generalizes a classical theorem of Bell (\cite{Bell81},  \cite{Bell82}; see also \cite{CF}):

\begin{thm}\label{BK transformation thm}
	Let $M_1$  and $M_2$ be two complex analytic sets.
	Let $V_1\subseteq M_1$ and $V_2 \subseteq M_2$ be proper analytic subvarieties such that $M_1-V_1, M_2-V_2$ are complex manifolds of the same dimension.
	Assume that $f:M_1-V_1\rightarrow M_2-V_2$ is a finite ($m-$sheeted) holomorphic covering map. Let $\Gamma$ be the deck transformation group for the covering map (with $|\Gamma|=m$), and denote by $K_i$ the Bergman kernels of $M_i$ for $i=1,2$. Then the Bergman kernel forms transform according to
	\begin{equation*}
	\sum_{\gamma\in \Gamma}(\id, \gamma)^*K_1=(f,f)^*K_2 \quad~\text{on}~(M_1-V_1) \times (M_1-V_1),
	\end{equation*}
	where $\id: M_1\rightarrow M_1$ is the identity map.
\end{thm}

\subsection{Finite ball quotients}
In this subsection, we recall the canonical realization of a finite ball quotient due to H. Cartan \cite{Cartan}.
Let $\mathbb{B}^n$ denote the unit ball in $\mathbb{C}^n$ and $\text{Aut}(\mathbb{B}^n)$ its (biholomorphic) automorphism group.
Let $\Gamma$ be a finite subgroup of $\text{Aut}(\mathbb{B}^n)$. Assume $\Gamma$ is fixed point free; that is, assume no $\gamma\in \Gamma-\{\id\}$ has any fixed points on $\partial\mathbb{B}^n$.
As the unitary group $U(n)$ is a maximal compact subgroup of $\Aut(\mathbb{B}^n)$, by basic Lie group theory, there exists some $\psi\in \Aut(\mathbb{B}^n)$ such that $\Gamma\subseteq \psi^{-1}\cdot U(n)\cdot \psi$. Thus without loss of generality, we can assume $\Gamma\subseteq U(n)$, i.e., $\Gamma$ is a finite unitary subgroup. The origin $0\in \mathbb{C}^n$ is then always a fixed point of every element in $\Gamma$. Moreover, the fixed point free condition on $\Gamma$ is equivalent to the assertion that every $\gamma\in \Gamma-\{\id\}$ has no other fixed point than $0$. We also note that, by the fixed point free condition, the action of $\Gamma$ on $\partial\mathbb{B}^n$ is properly discontinuous and $\partial\mathbb{B}^n/\Gamma$ is a smooth manifold.

By a theorem of H. Cartan \cite{Cartan}, the quotient $\mathbb{C}^n/\Gamma$ can be realized as a normal algebraic subvariety $V$ in some $\mathbb{C}^N$. To be more precise, we write $\mathcal{A}$ for the algebra of $\Gamma$-invariant holomorphic polynomials, that is,
\begin{equation*}
\mathcal{A}:=\big\{p\in \mathbb{C}[z_1,\cdots, z_n]: p\circ\gamma=p \,\mbox{ for all }\gamma \in\Gamma \big\}.
\end{equation*}
By Hilbert's basis theorem, $\mathcal{A}$ is finitely generated. Moreover, we can find a minimal set of homogeneous polynomials $\{p_1, \cdots, p_N\}\subseteq \mathcal A$ such that every $p\in \mathcal{A}$ can be expressed in the form
\begin{equation*}
p(z)=q(p_1(z), \cdots, p_N(z)),
\end{equation*}
where $q$ is some holomorphic polynomial in $\mathbb{C}^N$.
The map $Q:=(p_1,\cdots,p_N): \mathbb{C}^n\rightarrow \mathbb{C}^N$ is proper and induces a homeomorphism of $\mathbb{C}^n/\Gamma$ onto $V:=Q(\mathbb{C}^n)$. As $Q$ is a proper holomorphic polynomial map, $V$ is an algebraic variety. The restriction of $Q$ to the unit ball $\mathbb{B}^n$ maps $\mathbb{B}^n$ properly onto a relatively compact domain $\Omega\subseteq V$. In this way, $\mathbb{B}^n/\Gamma$ is realized as $\Omega$ by $Q$.
Following \cite{Rudin}, we call such $Q$ the \emph{basic map} associated to $\Gamma$.
The ball quotient $\Omega=\mathbb{B}^n/\Gamma$ is nonsingular if and only if the group $\Gamma$ is generated by \emph{reflections}, i.e., elements of finite order in $U(n)$ that fix a complex subspace of dimension $n-1$ in $\mathbb{C}^n$ (see \cite{Rudin}); thus, if $\Gamma$ is fixed point free and nontrivial, then $\Omega=\mathbb{B}^n/\Gamma$ must have singularities.
Moreover, $\Omega$ has smooth boundary if and only if $\Gamma$ is fixed point free (see \cite{Fo86} for more results along this line).

\section{Proof of Theorem \ref{Thm 1} and \ref{Thm 2}}\label{Sec Thm 1 and 2}
In this section, we prove that Theorem \ref{Thm 1} and \ref{Thm 2} follow from Theorem \ref{main thm}; see Section \ref{subsection proof of Thm 1} and \ref{subsection proof of Thm 2}, respectively.

\subsection{Proof of Theorem \ref{Thm 2}}\label{subsection proof of Thm 2}

The implication (i) $\implies$ (ii) in Theorem \ref{Thm 2} is trivial. We therefore only need to prove the converse. Let $G$ be as in Theorem \ref{Thm 2} and assume the conditions in (ii) hold. To prove (i), assuming that Theorem \ref{main thm} has been proved, we proceed in three steps.

\textbf{Step 1.} It follows from the assumption that the boundary $\partial G$ is strongly pseudoconvex. We first prove that the boundary $\partial G$ is indeed spherical.
Recall that a CR hypersurface $M$ of dimension $2n-1$ is said to be spherical if it is locally CR diffeomorphic, near every point, to an open piece of the unit sphere $S^{2n-1}\subseteq \bC^n$. To prove that $\partial G$ is spherical near a given boundary point $q\in \partial G$, one first uses the K\"ahler-Einstein assumption and the localization of the Bergman kernel (see \cite{Fe}, \cite{BoSj} and also see Proposition 3.1 in \cite{HuLi} for a detailed and nice proof) near $q$ to study the coefficients in Fefferman's expansion of the Bergman kernel of a smaller domain in $V$, which shares an open piece of its boundary with $G$ and is biholomorphic to a smoothly bounded strongly pseudoconvex domain in $\bC^n$.  In the two-dimensional case ($n=2$), one applies the argument in \cite{FuWo} (see Section 2 in \cite{FuWo}) to prove that the coefficient of the logarithmic term in Fefferman's expansion of the Bergman kernel vanishes to infinite order at in an open neighborhood of $q$. Using the (local) resolution of the Ramadanov Conjecture in $\bC^2$, as in \cite{FuWo}, one deduces that $\partial G$ is spherical. In the higher dimensional case ($n\geq 3$), one uses the argument in \cite{HX} (see the proof of Theorem 1.1 in \cite{HX}) to study the coefficient of the principle term (strong singularity) in Fefferman's expansion of the Bergman kernel and prove that every boundary point of $G$ is CR-umbilical, which implies that $\partial G$ is spherical. The detailed proof for step 1 is contained in \cite{HuLi} (see Theorem 1.1 in \cite{HuLi}). We will omit the proof here.

\textbf{Step 2.} In this step, we prove that $G$ is biholomorphic to a ball quotient $\mathbb{B}^n/\Gamma$ for some finite fixed point free subgroup $\Gamma \subseteq U(n)$.
Since we know $\partial G$ is spherical from Step 1 and $\partial G$ is contained in a real algebraic hypersurface in $\bC^N$, it follows from Corollary 3.3 in \cite{Hu} that $\partial G$ is CR equivalent to CR spherical space form $S^{2n-1}/\Gamma$ where $\Gamma \subseteq U(n)$ is as above. More precisely, there is an algebraic CR map $F: S^{2n-1}\rightarrow \partial G$, which is a finite covering map. From this one can further prove that $G$ is biholomorphic to $\mathbb{B}^n/\Gamma$. The proof of this is identical with Step 3 in Section 5 of \cite{EbenfeltXiaoXu2020algebraicity}. The general setting of \cite{EbenfeltXiaoXu2020algebraicity} is in dimension $n=2$, but, as pointed out in Remark 5.4 in \cite{EbenfeltXiaoXu2020algebraicity}, this argument works for all dimensions. The argument shows that $F$ extends to a proper, holomorphic branched covering map from $\mathbb{B}^n$ onto $G$, which realizes $G$ as the ball quotient $\mathbb{B}^n/\Gamma$. In particular, $G$ is biholomorphic to $\mathbb{B}^n/\Gamma$ as claimed. Since $\Gamma\subseteq U(n)$ is fixed point free, either $G$ has one unique singular point at $F(0)$ when $\Gamma\neq \{\id\}$ or $G$ is smooth when $\Gamma=\{\id\}$. In the former case, $F: \mathbb{B}^n-\{0\}\rightarrow G-\{F(0)\}$ is a smooth covering map whose group of deck transformations is $\Gamma$, and in the latter case, $F$ extends as a biholomorphism $\bB^n\to G$.

\textbf{Step 3.} By the conclusion in Step 2, the fundamental group of the regular part of $G$ is isomorphic to $\Gamma$. By assumption in (ii), $\Gamma$ is abelian. Moreover, the biholomorphism between $G$ and $\mathbb{B}^n/\Gamma$ gives an isometry between the Bergman metrics of $G$ and $\mathbb{B}^n/\Gamma$. By assumption in (ii) again, the Bergman metric of $\mathbb{B}^n/\Gamma$ is K\"ahler-Einstein. Thus, by Theorem \ref{main thm}, $\Gamma$ must then be the trivial group $\{\id\}$. Hence $G$ is biholomorphic to $\mathbb{B}^n$.
\qed

\subsection{Proof of Theorem \ref{Thm 1}}\label{subsection proof of Thm 1}
We now prove Theorem \ref{Thm 1}, under the assumption that Theorem \ref{main thm} has been proved. The "if" implication is trivial, and we only need to prove the converse. Thus, we assume that $G$ is as in Theorem \ref{Thm 1} and the Bergman metric of $G$ is K\"ahler-Einstein, and we shall prove $G$ is biholomorphic to $\mathbb{B}^n$. By copying the argument in Step 1 and Step 2 in Section \ref{subsection proof of Thm 2}, we conclude that there is an algebraic CR map $F$ from $S^{2n-1}$ to $\partial G\subseteq \partial\mathbb{B}^N=S^{2N-1}$, which is a finite covering map. In particular, the map $F$ induces a smooth, nonconstant CR map from the spherical space form $S^{2n-1}/\Gamma$, for some finite fixed point free subgroup  $\Gamma\subseteq \Aut(\mathbb{B}^n)$, to $S^{2N-1}$ (see \cite{Lich92}, \cite{DALi} and \cite{DA96}). Since $\Gamma$ is a finite subgroup of $\Aut(\mathbb{B}^n)$, by basic Lie group theory as above, $\Gamma$ is contained in a conjugate of the unitary group $U(n)$. By Theorem 8 in \cite{DALi}, 
$\Gamma$ is conjugate to one in a short list of special cyclic subgroups of $U(n)$. In particular, the finite subgroup $\Gamma$, as well as the fundamental group of the regular part of $G$ then, is abelian. Now, Theorem \ref{Thm 1} follows from Theorem \ref{Thm 2}.
\qed

\section{Proof of Theorem \ref{main thm}}\label{Sec main thm proof}
In this section, we shall prove Theorem \ref{main thm}. It suffices to prove the Bergman metric of $\bB^n /\Gamma$ cannot be K\"ahler-Einstein if $\Gamma \subseteq \mathrm{Aut}(\bB^n)$ is nontrivial, abelian, and fixed point free. We will prove this by contradiction. Thus, we suppose $\Gamma \subseteq \mathrm{Aut}(\bB^n)$ is abelian and fixed point free, $\Gamma \neq \{ \id \},$ and the Bergman kernel of $\Omega=\bB^n /\Gamma$ is K\"ahler-Einstein. As before, we know $\Gamma$ is contained in a conjugate of $U(n).$ Thus, without loss of generality, we will assume $\Gamma \subseteq U(n).$

We shall split our proof into three subsections.  Section 4.1 reduces the K\"ahler-Einstein condition of the Bergman metric to a functional equation (see equation \eqref{KE equation MA general}) for general finite, fixed point free groups $\Gamma\subset\Aut(\bB^n)$.  
In Section 4.2, we focus on the case where the group $\Gamma$ is additionally assumed to be abelian, and simplify the equation further into a rather explicit one (see equation (\ref{KE equation cyclic})). After that, in Section 4.3, we take the Taylor expansion of both sides of the equation. By carefully comparing the lowest order Taylor terms, we conclude that they can never match up due to some combinatorial inequalities. 
The proofs of these inequalities are then given in Section \ref{Sec combinatorial lemma}, which concludes the proof of Theorem \ref{main thm}.

\subsection{The K\"ahler-Einstein equation on finite ball quotients}
Since any two realizations of $\bB^n/\Gamma$
are biholomorphic, we can use H. Cartan's canonical realization of $\bB^n/\Gamma$, which was discussed Section 2.2. Thus, let $Q: \bC^n\rightarrow \bC^N$ be the basic map realizing $\bB/\Gamma$ as a domain $\Omega:=Q(\bB^n)$ in the $n$-dimensional algebraic variety $Q(\bC^n)$, as explained in Section 2.2. Set
\begin{equation*}
Z:=\{z\in \mathbb{C}^n: \mbox{the Jacobian of $Q$ at $z$ is not of full rank}  \}.
\end{equation*}
Note that in fact $Z=\{0\}$ by the fixed point free condition and nontriviality of $\Gamma$ (see \cite{Cartan} and \cite{EbenfeltXiaoXu2020algebraicity}). We denote by $K_{\Omega}$ and $K_{\bB^n}$ the Bergman kernel forms of $\Omega$ and $\bB^n$ respectively. By the transformation formula in Theorem \ref{BK transformation thm}, they are related by
\begin{equation}\label{BK transformation}
\sum_{\gamma\in \Gamma}(\id, \gamma)^*K_{\bB^n}=(Q,Q)^*K_{\Omega} \quad ~\text{on}~(\bB^n-Z) \times (\bB^n-Z),
\end{equation}
where $\id: \bB^n\rightarrow \bB^n$ is the identity map.
We note that
$$Q^*(i\partial \overline{\partial} \log \widetilde{K}_{\Omega})=i \partial \overline{\partial} \log ((Q,Q)^* \widetilde{K}_{\Omega}).$$
Furthermore, we also note that the K\"ahler-Einstein condition is a local property and that $Q$ is a local biholomorphism (on $\bB^n-Z$). It follows that the Bergman metric of $\Omega$ is K\"ahler-Einstein  if and only if the logarithm of the left hand side of (\ref{BK transformation}), restricted to the diagonal $w=z$, gives the potential function of a K\"ahler-Einstein metric on $\bB^n -Z.$

Recall the notation $\langle u, v\rangle=\sum_{i=1}^n u_i v_i$ for two column vectors $u=(u_1, \cdots, u_n)^{\intercal}, v=(v_1, \cdots, v_n)^{\intercal}$. Set $d_{\gamma}:=\det \gamma$ for $\gamma \in U(n)$. The left hand side of \eqref{BK transformation}, in the standard coordinates $z,w$ of $\bC^n$, equals
\begin{equation*}
	\sum_{\gamma\in \Gamma}(\id, \gamma)^*K_{\bB^n}=\frac{n!}{\pi^n}\sum_{\gamma\in \Gamma} \frac{\oo{d_{\gamma}}}{(1-\langle z,\oo{\gamma w}\rangle)^{n+1}}dz_1\wedge\ldots dz_n\wedge d\oo w_1\wedge\ldots\wedge d\oo w_n,
\end{equation*}
where $z, w \in \bB^n$ are regarded as column vectors and the elements of $\Gamma$ as unitary matrices. We introduce the function
\begin{equation*}
	\varphi(z,\oo{w}):=\sum_{\gamma\in \Gamma} \frac{\oo{d_{\gamma}}}{(1-\langle z,\oo{\gamma w}\rangle)^{n+1}},
\end{equation*}
and note that $\varphi(z, \oo{z})$ is real analytic on $\mathbb{B}^n.$ By the preceding discussion, we conclude that the Bergman metric of $\Omega$ is K\"ahler-Einstein if and only if $\varphi=\varphi(z,\bar z)$ is the potential function of a K\"ahler-Einstein metric, i.e., for $z\in \bB^n - Z$ and some constant $c_1\in \bR$,
\begin{equation}\label{KE equation general}
	-\partial \oo{\partial}\log G(z, \oo{z}) =-c_1\,\partial \oo{\partial} \log\varphi(z, \oo{z}),
\end{equation}
where $G=\det (g_{i\oo{j}})$ with $g_{i \oo{j}}=\partial_{z_i}\partial_{\oo{z_j}}\log\varphi$.
(We remark that one can use the result of Klembeck \cite{Kl}
to find the value of $c_1$, but this value will also come out directly from our arguments below.) The equation (\ref{KE equation general}) is equivalent to the statement that $\log G- c_1\log \varphi$ is pluriharmonic on $\bB^n-Z$. Consequently, since $Z=\{0\}$ and $n\geq 2$ so that $\bB^n -Z$ is simply connected, there exists some holomorphic function $h$ on $\bB^n-Z$ such that
\begin{equation*}
	\log G(z, \oo{z})- c_1\log \varphi(z, \oo{z})=h(z)+\oo{h(z)}.
\end{equation*}

By Hartogs's extension theorem, again since $n\geq 2$, we may assume $h$ is holomorphic on $\mathbb{B}^n$.
\begin{lemma}\label{lemmah} The function
	$h$ is constant. Furthermore, $h+\oo{h}=n\ln(n+1)$ and $c_1=1$.
\end{lemma}

\begin{proof}
This lemma is in fact proved in \cite{HuLi} using ideas from \cite{FuWo}. For the reader's convenience, we also sketch a proof here. We give a slightly different proof in order to avoid some tedious computations.

Set $g=e^{2h}$. Then $g$ is holomorphic in $\mathbb{B}^n$ and $|g|=e^{h+\bar{h}}>0$. We first study the boundary behavior of $g$.
	
	\textbf{Claim.} $\lim_{|z|\rightarrow 1}|g|=a$ for some constant $0\leq a\leq \infty$.
	\begin{proof}[Proof of the claim]
		Note that
		\begin{equation}\label{eq 1 for c}
		|g|=e^{h+\bar{h}}=\frac{G}{\varphi^{c_1}}.
		\end{equation}
		We also note that
		\begin{align}\label{eq 2 for c}
		\begin{split}
		\frac{n!}{\pi^n}\varphi(z,\bar{z})=&\frac{n!}{\pi^n}\Bigl(\frac{1}{(1-|z|^2)^{n+1}}+\sum_{\gamma\in \Gamma, \gamma\neq \id} \frac{\oo{d_{\gamma}}}{(1-\langle z,\oo{\gamma z} \rangle)^{n+1}}\Bigr)\\
		:=& \frac{n!}{\pi^n}\frac{1}{(1-|z|^2)^{n+1}}+T(z,\bar{z}),
		\end{split}			
		\end{align}
where $T(z,\bar{z})$ is real analytic in a neighborhood of $\oo{\mathbb{B}^n}$ since $\Gamma$ is assumed to be fixed point free. In particular, the asymptotic singular part of $\frac{n!}{\pi^n}\varphi$ as $z \rightarrow \partial \bB^n$ is the same as that of the Bergman kernel of $\mathbb{B}^n$. Let $J$ be the Monge-Amp\`ere type operator as defined in \eqref{MA operator}. With the preceding observation and the well known formula
		\begin{equation*}
		G=\det\Bigl(\partial_{z_i}\partial_{\oo{z_j}}\log\bigl(\varphi\bigr) \Bigr)=\frac{J(\varphi)}{\varphi^{n+1}},
		\end{equation*}
a simple calculation yields that the most singular part of $G$ (as $z \rightarrow \partial \bB^n$) is identical with that of the volume form of the Bergman metric on $\mathbb{B}^n$. More precisely,
		\begin{equation}\label{eq 3 for c}
		G=\frac{(n+1)^n}{(1-|z|^2)^{n+1}}+\widehat{G},
		\end{equation}
where $\widehat{G}$ is real analytic in $\mathbb{B}^n-Z$ and satisfies $(1-|z|^2)^{n+1}\widehat{G}\rightarrow 0$ as $|z|\rightarrow 1$. Then by \eqref{eq 1 for c}, \eqref{eq 2 for c} and \eqref{eq 3 for c}, we see
		\begin{equation}\label{eq 4 for c}
		\lim_{|z|\rightarrow 1} |g|=(n+1)^n\lim_{|z|\rightarrow 1} \frac{(1-|z|^2)^{(n+1){c_1}}}{(1-|z|^2)^{n+1}}.
		\end{equation}
Thus, depending on $c_1$, we have $\lim_{|z|\rightarrow 1} |g|=a$ for some $0\leq a\leq \infty$.
This proves the claim.
\end{proof}
But $g$ is a nowhere vanishing holomorphic function in $\mathbb{B}^n$. A standard maximum principle argument applied to $g$ and $\frac{1}{g},$ respectively, yields $a\neq 0$ and $a\neq \infty$, respectively. Hence $0<a<\infty$. But by \eqref{eq 4 for c}, this happens if and only if $c_1=1$. And in this case by \eqref{eq 4 for c}, $a=(n+1)^n$. Applying the maximum principle again, we see $|g|\equiv a=(n+1)^n.$
This implies $g$ and thus $h$ are constant functions, and $h+ \oo{h} \equiv n\ln (n+1).$ The proof of the lemma is finished.
\end{proof}

We define the Monge-Amp\`ere type operator $J$ as follows (note that it differs by a sign from the standard operator introduced by Fefferman \cite{Fe2}): 
\begin{equation}\label{MA operator}
	J(\varphi):=\det
	\begin{pmatrix}
	\varphi & \varphi_{\oo{z_j}}\\
	\varphi_{z_i} & \varphi_{z_i\oo{z_j}}
	\end{pmatrix},
\end{equation}
We use Lemma \ref{lemmah} and the well-known formula $G=J(\varphi)/\varphi^{n+1}$ to further simplify \eqref{KE equation general} into
\begin{equation}\label{KE equation MA general}
	J(\varphi)(z, \overline{z})=(n+1)^n\varphi^{n+2}(z, \oo{z})
\end{equation}
for $z \in \bB^n-Z.$ Since both sides of \eqref{KE equation MA general} are in fact real-analytic in $\bB^n,$ we see \eqref{KE equation MA general} holds on $\bB^n$ by continuity. We pause here to observe that if $\Gamma$ is such that $\varphi(0,0)\neq 0$, then it follows that $\log \varphi$ extends as the potential of a K\"ahler-Einstein metric in the whole unit ball $\bB^n$, which by uniqueness of the Cheng-Yau metric can be used to directly conclude that $\Gamma=\{\id\}$; this was previously observed in \cite[Corollary 5.4]{HuLi}. Now, let us compute $J(\varphi)$. Clearly, we have
\begin{align*}
	\varphi_{z_i}=(n+1)\sum_{\gamma\in \Gamma} \frac{ \oo{d_{\gamma}} \cdot \oo{(\gamma z)_i}}{(1-\langle z, \oo{\gamma z}\rangle)^{n+2}},
	\qquad
	\varphi_{\oo{z_j}}=(n+1)\sum_{\gamma\in \Gamma} \frac{ \oo{d_{\gamma}}\cdot (z^{\intercal}\oo{\gamma})_j}{(1-\langle z, \oo{\gamma z}\rangle)^{n+2}},
\end{align*}
where $(\gamma z)_i$ denotes the $i$-th entry of the column vector $\gamma z$ and similarly $(z^{\intercal}\oo{\gamma})_j$ denotes the $j$-th entry of the row vector $z^{\intercal}\oo{\gamma}$.
By differentiating both sides one more time, we obtain
\begin{align*}
	\varphi_{z_i\oo{z_j}}
	=&(n+1)\sum_{\gamma\in \Gamma} \oo{d_{\gamma}}\cdot\frac{  \oo{\gamma_{ij}}(1-\langle z, \oo{\gamma z}\rangle)+(n+2)\oo{(\gamma z)_i} (z^{\intercal}\oo{\gamma})_j}{(1-\langle z, \oo{\gamma z}\rangle)^{n+3}},
\end{align*}
where $\gamma_{ij}$ is the $(i,j)$ component of the matrix $\gamma$.

For each $\gamma\in \Gamma, 0 \leq j \leq n$, we define a column vector-valued function  $\xi_j(\gamma): \mathbb{B}^n \rightarrow \mathbb{C}^{n+1}$  in the variables $(z, \oo{z})$ as follows:
\begin{align*}
	\xi_0(\gamma)(z,\oo z):=\begin{pmatrix}
	1-\langle z, \gamma\oo{z} \rangle\\
	(n+1)\gamma\oo{z}
	\end{pmatrix}
	\quad \mbox{and} \quad \xi_j(\gamma)(z,\oo z):=\begin{pmatrix}
		z^{\intercal}(\gamma)_j\\
		\frac{  (\gamma)_j(1-\langle z, \gamma\oo{ z}\rangle)+(n+2)\gamma\oo{z} (z^{\intercal}(\gamma)_j)}{1-\langle z, \gamma\oo{z}\rangle}
	\end{pmatrix}  \mbox{ for } 1\leq j\leq n,
\end{align*}
where $(\gamma)_j$ is the $j$-th column vector of the matrix $\gamma$.
Given any $(n+1)$ (possibly repeated) elements $\gamma_0,\cdots,\gamma_n$ in $\Gamma$, we define a matrix-valued function $A(\gamma_0,\cdots,\gamma_n): \mathbb{B}^n \rightarrow \mathbb{C}^{(n+1)^2}$ as follows:
\begin{equation*}
	A(\gamma_0,\cdots,\gamma_n)=\begin{pmatrix}
	\xi_0(\gamma_0) & \cdots & \xi_n(\gamma_n)
	\end{pmatrix}.
\end{equation*}
We emphasize that the map $A(\gamma_0,\cdots,\gamma_n)$ sends a point $z \in \bB^n$ to an $(n+1) \times (n+1)$ matrix. We then expand the determinant in \eqref{MA operator} by multi-linearity with respect to columns. We obtain the following formula:
\begin{align}\label{eqnjphi}
	J(\varphi)=\sum_{\gamma_0,\cdots,\gamma_n\in \Gamma}\frac{(n+1)^n\,\oo{d_{\gamma_0}}\cdots\oo{d_{\gamma_n}}}{\prod_{i=0}^n(1-\langle z, \oo{\gamma_i z}\rangle)^{n+2}} \,\det\bigl(A(\oo{\gamma_0},\cdots,\oo{\gamma_n})\bigr).
\end{align}


\subsection{Abelian group case}
From now on, we will assume that $\Gamma$ is a finite, abelian, fixed point free subgroup of $U(n)$.
\begin{lemma}\label{lem gamma is cyclic}
	If $\Gamma \subseteq U(n)$ is finite, abelian, and fixed point free, then it is cyclic.
\end{lemma}

\begin{proof}
	Since $\Gamma\subseteq U(n)$ is abelian, by basic Lie group theory, $\Gamma$ is contained in a conjugate of the maximal torus $U(1)\times \cdots \times U(1)$. Replacing $\Gamma$ by an appropriate conjugate of $\Gamma$, we can assume $\Gamma\subseteq U(1)\times \cdots \times U(1)$. Consider the group homomorphism $\pi_1: \Gamma \rightarrow U(1)$ defined by
	\begin{equation*}
		\pi_1(\gamma):=\gamma_{11} \quad \mbox{ for } \gamma=\text{diag}(\gamma_{11},\cdots,\gamma_{nn}).
	\end{equation*}
Since $\pi_1(\Gamma)\subseteq U(1)$ is finite, it must be cyclic. To verify that $\Gamma$ is also cyclic, it is sufficient to prove that $\Gamma$ is actually isomorphic to $\pi_1(\Gamma)$. We will conclude this by showing that $\pi_1$ is injective. Thus, take $\gamma\in \Gamma$ such that
	\begin{equation*}
		\pi_1(\gamma)=1.
	\end{equation*}
	It follows that the point $(1,0,\cdots,0)$ is a fixed point of $\gamma$. Since $\Gamma$ is fixed point free, $\gamma$ is the identity matrix. So $\pi_1$ is injective and the proof is complete.
\end{proof}

By the proof of Lemma \ref{lem gamma is cyclic}, we can assume $\Gamma\subseteq U(1)\times \cdots \times U(1)$. As $\Gamma$ is actually cyclic, we can write
\begin{equation*}
\Gamma=\{\gamma, \gamma^2,\cdots,\gamma^m=\id\}
\end{equation*}
for some generator
\begin{equation*}
\gamma=\begin{pmatrix}
\varepsilon_1 & &\\
& \ddots &\\
& & \varepsilon_n
\end{pmatrix}.
\end{equation*}
Here $m \geq 2$ by the nontriviality of $\Gamma.$ Since $\Gamma$ is fixed point free, $\varepsilon_1,\cdots, \varepsilon_n$ are primitive $m$-th roots of unity. By setting $\varepsilon:=\varepsilon_1$, for $1\leq j\leq n$ we can write $\varepsilon_j$ in the form of
\begin{equation*}
\varepsilon_{j}=\varepsilon^{t_j}, \quad \mbox{for some } 1\leq t_j\leq m-1 \mbox{ with }\gcd(t_j, m)=1.
\end{equation*}
Without loss of generality, we can assume
\begin{equation*}
	1=t_1\leq t_2\leq \cdots \leq t_n\leq m-1.
\end{equation*}

For any $\gamma\in \Gamma\subseteq U(1)\times \cdots \times U(1)$, note that $\gamma^{-1}=\oo{\gamma}^\intercal=\oo{\gamma}$. Hence, we can replace all $\oo{\gamma_j}$ by $\gamma_j$ in the sum in $J(\varphi)$ and obtain
\begin{align*}
J(\varphi)=\sum_{\gamma_0,\cdots,\gamma_n\in \Gamma}\frac{(n+1)^n\,d_{\gamma_0}\cdots d_{\gamma_n}}{\prod_{i=0}^n(1-\langle z, \gamma_i\oo{ z}\rangle)^{n+2}}\,
\det\bigl(A(\gamma_0,\cdots,\gamma_n)\bigr).
\end{align*}

Write $\gamma_j=\gamma^{k_j}$ for some $0\leq k_j\leq m-1$. Then $d_{\gamma_{j}}=\det \gamma_j=\varepsilon^{k_j(\sum_{i=1}^n t_i)}.$ Choose $z=z^*:=(z_1,0,\cdots,0)^{\intercal}$ with $ |z_1| < 1$ and set $x=z^*\cdot \oo{z^*}=z_1\oo{z_1}< 1$. Then at $z^*,$ we have
\begin{align*}
	\det\bigl(A(\gamma_0,\cdots,\gamma_n)\bigr)=&\begin{vmatrix}
	1-\varepsilon^{k_0}x & z_1\varepsilon^{k_1}& 0\\
	(n+1)\varepsilon^{k_0}\oo{w_1} & \frac{  \varepsilon^{k_1}(1-\varepsilon^{k_1}x)+(n+2)\varepsilon^{2k_1}x}{1-\varepsilon^{k_1}x} & 0 \\
	0&0&\begin{pmatrix}
	\varepsilon_2^{k_2} & &\\
	& \ddots &\\
	& & \varepsilon_n^{k_n}
	\end{pmatrix}
	\end{vmatrix}
	\\
	=&\varepsilon^{k_1+\sum_{j=2}^nk_2t_2}\Bigl( 1-(n+2)\varepsilon^{k_0}x+(n+2)\frac{1-\varepsilon^{k_0}x}{1-\varepsilon^{k_1}x}\,\varepsilon^{k_1}x \Bigr).
\end{align*}

In the following, we use the notation
\begin{itemize}
	\item $T=(t_1,\cdots, t_n)$, where $t_1=1$.
	\item $K=(k_0, k_1, \cdots, k_n)$ and $K'=(k_1,\cdots, k_n)$.
	\item $|T|=\sum_{j=1}^{n}t_j$ and $|K'|=\sum_{j=1}^nk_j$.
\end{itemize}
Using these notations, we have, at $z=z^*$,
\begin{align*}
	J(\varphi)(z^*, \oo{z^*})=
	(n+1)^n\sum_{k_0,\cdots,k_n=0}^{m-1}\frac{\varepsilon^{|K|\cdot|T|+K'\cdot T}}{\prod_{i=0}^n(1-\varepsilon^{k_i}x)^{n+2}}
	\Bigl(1-(n+2)\varepsilon^{k_0}x+(n+2)\frac{1-\varepsilon^{k_0}x}{1-\varepsilon^{k_1}x}\,\varepsilon^{k_1}x \Bigr).
\end{align*}
If we set
\begin{align*}
	\text{I}:=&\sum_{k_0,\cdots,k_n=0}^{m-1}\frac{\varepsilon^{|K|\cdot|T|+K'\cdot T}}{\prod_{i=0}^n(1-\varepsilon^{k_i}x)^{n+2}},
	\\
	\text{II}:=&-(n+2)\sum_{k_0,\cdots,k_n=0}^{m-1}\frac{\varepsilon^{|K|\cdot|T|+K'\cdot T+k_0}x}{\prod_{i=0}^n(1-\varepsilon^{k_i}x)^{n+2}},
	\\
	\text{III}:=&(n+2)\sum_{k_0,\cdots,k_n=0}^{m-1}\frac{\varepsilon^{|K|\cdot|T|+K'\cdot T+k_1}x}{\prod_{i=0}^n(1-\varepsilon^{k_i}x)^{n+2}}
	\frac{1-\varepsilon^{k_0}x}{1-\varepsilon^{k_1}x},
\end{align*}
then
\begin{equation}\label{J(phi) in three terms}
	J(\varphi)(z^*, \oo{z^*})=(n+1)^n\bigl(\text{I}+\text{II}+\text{III}\bigr).
\end{equation}

We pause to introduce the following definition and lemmas. Let $\varepsilon$ be as above.  Write $\mathbb{D}$ for the open unit disc in $\mathbb{C}.$
\begin{defn}
	Let $t\in \mathbb{Z}, p\in \mathbb{Z}^+$. Define $f_{t,p}: \mathbb{D} \rightarrow\bC$ as
	\begin{equation*}
		f_{t,p}(x):=\sum_{k=0}^{m-1}\frac{1}{\varepsilon^{tk}(\varepsilon^k-x)^p}.
	\end{equation*}
\end{defn}

\begin{lemma}\label{lemma 2 for f} The following holds:
	\begin{equation*}
		f_{t,p}'(x)=pf_{t,p+1}(x).
	\end{equation*}
In general, for $j \geq 2,$ $f_{t,p}^{(j)}(x)=p(p+1) \cdots (p+j-1) f_{t, p+j}(x).$
\end{lemma}
\begin{proof}Note
	\begin{align*}
		f_{t,p}'(x)=\sum_{k=0}^{m-1} \frac{p}{\varepsilon^{tk}(\varepsilon^k-x)^{p+1}}=pf_{t,p+1}(x).
	\end{align*}
This proves the first statement. The latter assertion follows from the first statement and an inductive argument.
\end{proof}

\begin{lemma}\label{lemma 1 for f} The following hold:
\begin{enumerate}
	\item
\begin{equation*}
	f_{t,p}(0)=\begin{cases}
	0 & \mbox{ if } m \nmid (t+p),\\
	m & \mbox{ if } m \mid (t+p).
	\end{cases}
	\end{equation*}
\item For $j \geq 1,$
\begin{equation*}
f_{t,p}^{(j)}(0)=\begin{cases}
0 & \mbox{ if } m \nmid (t+p+j),\\
m\prod_{i=0}^{i=j-1}(p+i) & \mbox{ if } m \mid (t+p+j).
\end{cases}
\end{equation*}
\end{enumerate}
\end{lemma}
\begin{proof} To prove part (1), we note that
	\begin{equation*}
	f_{t,p}(0)=\sum_{k=0}^{m-1}\varepsilon^{-k(t+p)}.
	\end{equation*}
	Then the result in part (1) follows directly by the fact that $\varepsilon$ is a primitive $m$-th root of unity.
Part (2) follows from part (1) and Lemma \ref{lemma 2 for f}.
\end{proof}

\begin{lemma}\label{lemma46} The following holds:
	\begin{equation*}
		\sum_{k=0}^{m-1}\frac{\varepsilon^{t k}}{(1-\varepsilon^kx)^p}=f_{t-p,p}(x).
	\end{equation*}
\end{lemma}
\begin{proof}
	\begin{align*}
		\sum_{k=0}^{m-1}\frac{\varepsilon^{t k}}{(1-\varepsilon^kx)^p}
		=\sum_{k=0}^{m-1}\frac{1}{\varepsilon^{t k}(1-\varepsilon^{-k}x)^p}
		=\sum_{k=0}^{m-1}\frac{1}{\varepsilon^{(t-p) k}(\varepsilon^k-x)^p}=f_{t-p,p}(x),
	\end{align*}
	where the first equality follows from the fact that $\varepsilon$ is a primitive $m$-th root of unity.
\end{proof}

Now, using the above notation and Lemma \ref{lemma46}, we shall express $J(\varphi)(z^*, \oo{z^*})$ in terms of $f_{t,p}$.
\begin{flalign*}
	\quad\text{I}
	=&\sum_{k_0,\cdots,k_n=0}^{m-1}\frac{\varepsilon^{|K|\cdot|T|+K'\cdot T}}{\prod_{i=0}^n(1-\varepsilon^{k_i}x)^{n+2}}\\
	=&\sum_{k_0=0}^{m-1}\frac{\varepsilon^{k_0|T|}}{(1-\varepsilon^{k_0}x)^{n+2}}\, \sum_{k_1=0}^{m-1}\frac{\varepsilon^{k_1(|T|+t_1)}}{(1-\varepsilon^{k_1}x)^{n+2}}\cdots \sum_{k_n=0}^{m-1}\frac{\varepsilon^{k_n(|T|+t_n)}}{(1-\varepsilon^{k_n}x)^{n+2}}\\
	=&f_{|T|-(n+2), n+2}(x)\, f_{|T|+t_1-(n+2), n+2}(x)\cdots f_{|T|+t_n-(n+2), n+2}(x). &&
\end{flalign*}
\begin{flalign*}
	\quad\text{II}=&-(n+2)x\sum_{k_0,\cdots,k_n=0}^{m-1}\frac{\varepsilon^{|K|\cdot|T|+K'\cdot T+k_0}}{\prod_{i=0}^n(1-\varepsilon^{k_i}x)^{n+2}}\\
	=&-(n+2)x\sum_{k_0=0}^{m-1}\frac{\varepsilon^{k_0(|T|+1)}}{(1-\varepsilon^{k_0}x)^{n+2}}\, \sum_{k_1=0}^{m-1}\frac{\varepsilon^{k_1(|T|+t_1)}}{(1-\varepsilon^{k_1}x)^{n+2}}\cdots \sum_{k_n=0}^{m-1}\frac{\varepsilon^{k_n(|T|+t_n)}}{(1-\varepsilon^{k_n}x)^{n+2}}\\
	=&-(n+2)x\,f_{|T|-(n+1), n+2}(x)\, f_{|T|+t_1-(n+2), n+2}(x)\cdots f_{|T|+t_n-(n+2), n+2}(x). &&
\end{flalign*}
\begin{flalign*}
	\quad\text{III}=&(n+2)x\sum_{k_0,\cdots,k_n=0}^{m-1}\frac{\varepsilon^{|K|\cdot|T|+K'\cdot T+k_1}}{\prod_{i=0}^n(1-\varepsilon^{k_i}x)^{n+2}}
	\frac{1-\varepsilon^{k_0}x}{1-\varepsilon^{k_1}x}\\
	=&(n+2)x\sum_{k_0=0}^{m-1}\frac{\varepsilon^{k_0|T|}}{(1-\varepsilon^{k_0}x)^{n+1}}\, \sum_{k_1=0}^{m-1}\frac{\varepsilon^{k_1(|T|+t_1+1)}}{(1-\varepsilon^{k_1}x)^{n+3}} \,
	\sum_{k_2=0}^{m-1}\frac{\varepsilon^{k_n(|T|+t_2)}}{(1-\varepsilon^{k_2}x)^{n+2}}
	\cdots \sum_{k_n=0}^{m-1}\frac{\varepsilon^{k_n(|T|+t_n)}}{(1-\varepsilon^{k_n}x)^{n+2}}\\
	=&(n+2)x\,f_{|T|-(n+1), n+1}(x)\, f_{|T|+t_1-(n+2), n+3}(x)\, f_{|T|+t_2-(n+2), n+2}(x)\cdots f_{|T|+t_n-(n+2), n+2}(x). &&
\end{flalign*}

Set
\begin{align}\label{P and Q}
\begin{split}
P:=&f_{|T|-(n+2), n+2} f_{|T|-(n+1), n+2}-(n+2)x \bigl(f_{|T|-(n+1), n+2}^2-f_{|T|-(n+1), n+1} f_{|T|-(n+1), n+3}\bigr),\\
Q:=&f_{|T|+t_2-(n+2), n+2}\cdots f_{|T|+t_n-(n+2), n+2}.
\end{split}
\end{align}
By \eqref{J(phi) in three terms} and the fact $t_1=1$, we conclude that $J(\varphi)(z^*,\oo{z^*})$ can be written as
\begin{align}\label{eqnjpq}
	J(\varphi)(z^*, \oo{z^*})=(n+1)^nP(x)\,Q(x).
\end{align}

Moreover, at $z=z^*=(z_1,0,\cdots,0)^{\intercal}$ we can simplify $\varphi$ as
\begin{align*}
	\varphi(z^*, \oo{z^*})=\sum_{\gamma\in \Gamma} \frac{\oo{d_{\gamma}}}{(1-\langle z^*, \oo{\gamma z^*}\rangle)^{n+1}}
	=\sum_{\gamma\in \Gamma}\frac{d_{\gamma}}{(1-\langle z^*, \gamma \oo{z^*}\rangle)^{n+1}}=\sum_{k=0}^{m-1}\frac{\varepsilon^{k\,|T|}}{(1-\varepsilon^k x)^{n+1}}=f_{|T|-(n+1),n+1}(x).
\end{align*}
The second equality here is due to the fact that $\Gamma\subseteq U(1)\times\cdots\times U(1)\subseteq U(n)$, as also explained above.
By the above expression for $\varphi$ and \eqref{eqnjpq}, we conclude that at $z=z^*,$ the K\"ahler-Einstein equation \eqref{KE equation MA general} is reduced to, for $x \in [0, 1) \subseteq \mathbb{R},$ 
\begin{equation}\label{KE equation cyclic}
	f_{|T|-(n+1),n+1}^{n+2}(x)=P(x) Q(x),
\end{equation}
where $P, Q$ are defined in \eqref{P and Q}. Since both sides of \eqref{KE equation cyclic} are holomorphic in $\mathbb{D},$ we conclude that \eqref{KE equation cyclic} in fact holds for all $x \in \mathbb{D}.$

\subsection{Reduction to combinatorial inequalities}
We shall take the Taylor expansion of both sides in \eqref{KE equation cyclic} at $x=0$. By comparing the Taylor coefficients, we shall prove that \eqref{KE equation cyclic} cannot hold if $m=|\Gamma|\geq 2$ and $n\geq 2$, which will establish Theorem \ref{main thm}. We shall proceed by dividing the proof into several cases.

\textbf{Case I.} $m \mid |T|$.

As $m\mid |T|$, $m \nmid |T|+1$. Lemma \ref{lemma 1 for f} yields that
\begin{equation*}
	f_{|T|-(n+1),n+1}(0)=m, \qquad f_{|T|-(n+1),n+2}(0)=0.
\end{equation*}
Therefore, at $x=0$
\begin{equation*}
	f_{|T|-(n+1),n+1}^{n+2}(0)=m^{n+2}\neq 0=P(0)\cdot Q(0),
\end{equation*}
which implies that the K\"ahler-Einstein equation \eqref{KE equation cyclic} does not hold.

\textbf{Case II.} $m \nmid |T|$ and $m\mid |T|+1$.

In this case, we have
\begin{equation*}
	m\nmid |T|+2, \cdots, m\nmid |T|+m,\, m\mid |T|+m+1.
\end{equation*}
We take the Taylor expansion of $f_{|T|-(n+2), n+2}$ at $x=0$.
\begin{align}\label{eqnty}
\begin{split}
	f_{|T|-(n+2),n+2}(x)
	=&\sum_{j=0}^{m+1}\frac{f_{|T|-(n+2),n+2}^{(j)}(0)}{j!}x^j+O({m+2})\\
	=&\binom{n+2}{1}mx+\binom{n+m+2}{m+1}mx^{m+1}+O({m+2}).
\end{split}
\end{align}
Here for a holomorphic function $h$ in a neighborhood $U \subseteq \mathbb{C}$ of $0$, we say $h$ is $O(j), j \geq 1, $ if $h^{(i)}(0)=0$ for all $0 \leq i < j.$
The last equality follows from Lemma \ref{lemma 1 for f}.

Similarly, we also have
\begin{align*}
	f_{|T|-(n+1),n+2}(x)
	=&m+\binom{m+n+1}{m}mx^m+O({m+1}),
	\\
	f_{|T|-(n+1),n+1}(x)
	=&\binom{n+1}{1}mx+\binom{m+n+1}{m+1}mx^{m+1}+O({m+2}),
	\\
	f_{|T|-(n+1),n+3}(x)
	=&\binom{m+n+1}{m-1}mx^{m-1}+O({m}).
\end{align*}
By \eqref{P and Q}, it follows that
\begin{align*}
	P=&	\left(\binom{n+2}{1}mx+\binom{n+m+2}{m+1}mx^{m+1}\right)\left(m+\binom{m+n+1}{m}mx^m \right)
	\\&-(n+2)x\left(m+\binom{m+n+1}{m}mx^m \right)^2
	\\&+(n+2)x\left(\binom{n+1}{1}mx+\binom{m+n+1}{m+1}mx^{m+1}\right)\binom{m+n+1}{m-1}mx^{m-1}+O({m+2})
	\\
	=&m^2\binom{n+m+1}{m}x^{m+1}\left(-n-2+\frac{m+n+2}{m+1}+(n+1)m\right)+O({m+2})
	\\
	=&\frac{m^4(n+1)}{m+1}\binom{n+m+1}{m}x^{m+1}+O({m+2}).
\end{align*}

Recall that
\begin{equation*}
	Q=f_{|T|+t_2-(n+2), n+2}\cdots f_{|T|+t_n-(n+2), n+2},
\end{equation*}
where $1=t_1\leq t_2\leq \cdots \leq t_n\leq m-1$.
Let $1\leq a\leq n$ be such that
\begin{equation*}
	1=t_1=\cdots=t_a<t_{a+1}\leq \cdots\leq t_n.
\end{equation*}
When $a=n,$ the above means that all $t_j'$s equal $1$. Now for $1\leq j\leq a$, we have
\begin{equation*}
	f_{|T|+t_j-(n+2),n+2}(x)=f_{|T|-(n+1),n+2}(x)=m+O(1).
\end{equation*}
And for $a+1\leq j\leq n$, $|T|+1<|T|+t_j<|T|+m+1$. We get, by a similar computation as in \eqref{eqnty},
\begin{equation*}
	f_{|T|+t_j-(n+2),n+2}(x)=\binom{n+m+2-t_j}{m+1-t_j}m x^{m+1-t_j}+O({m+2-t_j}).
\end{equation*}
Thus,
\begin{align*}
	Q=m^{n-1}\binom{m+n+2-t_{a+1}}{m+1-t_{a+1}}\cdots \binom{m+n+2-t_{n}}{m+1-t_{n}} x^{(m+1)(n-a)-(t_{a+1}+\cdots+t_n)}+\text{h.o.t},
\end{align*}
where h.o.t denotes the higher order term.
Combining this with the Taylor expansion of $P$, the lowest order term in $PQ$ at $x=0$ is
\begin{equation}\label{eq:PQat0}
	\frac{m^{n+3}(n+1)}{m+1}\binom{n+m+1}{m}\binom{m+n+2-t_{a+1}}{m+1-t_{a+1}}\cdots \binom{m+n+2-t_{n}}{m+1-t_{n}}x^{(m+1)(n-a+1)-\sum_{j=a+1}^nt_j}.
\end{equation}

When $a=n, \sum_{j=a+1}^nt_j$ is a null sum and equals zero. Furthermore, we have
\begin{equation}\label{eq:fblahat0-1}
	f^{n+2}_{|T|-(n+1),n+1}=(n+1)^{n+2}m^{n+2}x^{n+2}+\text{h.o.t}\,.
\end{equation}
Suppose that the K\"ahler-Einstein equation \eqref{KE equation cyclic} holds. Then $f^{n+2}_{|T|-(n+1),n+1}$ and $PQ$ must share the same Taylor expansion at $x=0$. In particular, their lowest order terms, where the former is found in \eqref{eq:fblahat0-1} and the latter in \eqref{eq:PQat0}, must have the same degree, that is, $n+2=(m+1)(n-a)-\sum_{j=a+1}^nt_j$. In this case, however, the coefficients of the lowest order terms do not match by the following lemma.
\begin{lemma}\label{lemma combinatorial 1}
	Suppose $m, n\geq 2$, $1\leq a\leq n$ and $1=t_1=\cdots=t_a<t_{a+1}\leq \cdots\leq t_n\leq m-1$. If $n+2=(m+1)(n-a+1)-\sum_{j=a+1}^n t_j$, then
	\begin{equation*}
		(n+1)^{n+1}(m+1)>m\binom{n+m+1}{m}\binom{m+n+2-t_{a+1}}{m+1-t_{a+1}}\cdots \binom{m+n+2-t_{n}}{m+1-t_{n}}.
	\end{equation*}
In the case $a=n,$ i.e., all $t_j'$s equal $1$, the above is reduced to the following: If $m=n+1,$ then
$$(n+1)^{n+1}(m+1)>m\binom{n+m+1}{m}.$$
\end{lemma}

This is a contradiction and we thus conclude the K\"ahler-Einstein equation \eqref{KE equation cyclic} does not hold. The proof of Lemma \ref{lemma combinatorial 1} is left to Section \ref{Sec combinatorial lemma}.

\textbf{Case III.} $m\nmid |T|, m\nmid |T|+1,\cdots, m\nmid |T|+k-1$ and $m\mid |T|+k$ for some $2\leq k<m$.

We follow the same procedure as in Case II. Similarly as in \eqref{eqnty}, by using Lemma \ref{lemma 1 for f}, we have
\begin{align*}
f_{|T|-(n+2),n+2}(x)
=&\sum_{j=0}^{k+m}\frac{f_{|T|-(n+2),n+2}^{(j)}(0)}{j!}x^j+O(x^{k+m+1})\\
=&\binom{n+k+1}{k}mx^k+\binom{n+k+m+1}{k+m}mx^{k+m}+O({k+m+1}),
\end{align*}
and
\begin{align*}
f_{|T|-(n+1),n+2}(x)
=&\binom{n+k}{k-1}mx^{k-1}+\binom{k+m+n}{k+m-1}mx^{k+m-1}+O({k+m}),
\\
f_{|T|-(n+1),n+1}(x)
=&\binom{n+k}{k}mx^k+\binom{n+k+m}{k+m}mx^{k+m}+O({k+m+1}),
\\
f_{|T|-(n+1),n+3}(x)
=&\binom{n+k}{k-2}mx^{k-2}+\binom{n+k+m}{k+m-2}mx^{k+m-2}+O({k+m-1}).
\end{align*}
By \eqref{P and Q}, it follows that
\begin{align}\label{P for case III}
P=\binom{n+k}{k-1}\binom{n+k+m}{n}\frac{m^4}{k}x^{2k+m-1}+\text{h.o.t}.
\end{align}

Now we turn to the computation of the leading term in $Q$. Recall that $1=t_1\leq t_2\leq \cdots \leq t_n\leq m-1$. We shall divide the computation into two subcases: $k<t_n$ and $k\geq t_n$.

\textbf{Subcase III (a).} $k<t_n$.

Since $k\geq 2$, there exists some $1\leq a \leq n-1$ such that
\begin{equation*}
	1=t_1\leq \cdots \leq t_a\leq k<t_{a+1}\leq \cdots \leq t_n\leq m-1.
\end{equation*}
For $1\leq j\leq a$, as $|T|+1\leq |T|+t_j\leq |T|+k$, by the Taylor expansion and Lemma \ref{lemma 1 for f}, we have
\begin{equation*}
	f_{|T|+t_j-(n+2),n+2}(x)=\binom{n+1+k-t_j}{k-t_j}mx^{k-t_j}+O({k-t_j+1}).
\end{equation*}
For $a+1\leq j\leq n$, it follows that $|T|+k<|T|+t_j< |T|+k+m$. Thus, by the Taylor expansion and Lemma \ref{lemma 1 for f},
\begin{equation*}
 f_{|T|+t_j-(n+2),n+2}(x)=\binom{n+1+m+k-t_j}{m+k-t_j}mx^{m+k-t_j}+O({m+k-t_j+1}).
\end{equation*}
By \eqref{P and Q}, we obtain
\begin{equation*}
	Q=\prod_{j=2}^a\binom{n+1+k-t_j}{k-t_j}mx^{k-t_j} \cdot \prod_{j=a+1}^n\binom{n+1+m+k-t_j}{m+k-t_j}mx^{m+k-t_j}+\text{h.o.t}.
\end{equation*}
Therefore, \eqref{P for case III} and the above equality yield the leading term in the Taylor expansion of $PQ$ at $x=0$ as
\begin{align*}
	\frac{m^{n+3}}{k}\binom{n+k}{k-1}\binom{n+k+m}{n}	\prod_{j=2}^a\binom{n+1+k-t_j}{k-t_j}\cdot \prod_{j=a+1}^n\binom{n+1+m+k-t_j}{m+k-t_j}\cdot x^s,
\end{align*}
where
\begin{equation}\label{eq:seq}
	s=(n+1)k+m-1-\sum_{j=2}^at_j+\sum_{j=a+1}^n(m-t_j).
\end{equation}
On the other hand, the left hand side of \eqref{KE equation cyclic} satisfies
\begin{equation*}
	f_{|T|-(n+1),n+1}^{n+2}=\binom{n+k}{k}^{n+2}m^{n+2}x^{k(n+2)}+\text{h.o.t}.
\end{equation*}

Suppose that the K\"ahler-Einstein equation \eqref{KE equation cyclic} holds. Then $f^{n+2}_{|T|-(n+1),n+1}$ and $PQ$ must share the same Taylor expansion at $x=0$. In particular, their lowest order terms have the same degree, that is, $s=k(n+2)$, which in view of \eqref{eq:seq} implies that
\begin{equation}\label{eqncot}
		k=m-\sum_{j=1}^at_j+\sum_{j=a+1}^n(m-t_j).
	\end{equation}
 In this case, the coefficients of the lowest terms are, however, unequal by the following lemma.
\begin{lemma}\label{lemma combinatorial 2}
	Suppose $1\leq a\leq n-1$, $2\leq k\leq m-1$ and $1=t_1\leq \cdots \leq t_a\leq k<t_{a+1}\leq \cdots \leq t_n\leq m-1$. If \eqref{eqncot} holds, then
	\begin{align*}
		k\,\binom{n+k}{k}^{n+2}>m\,\binom{n+k+m}{n}	\prod_{j=1}^a\binom{n+1+k-t_j}{k-t_j}\cdot \prod_{j=a+1}^n\binom{n+1+m+k-t_j}{m+k-t_j}.
	\end{align*}
\end{lemma}

This is a contradiction and we thus conclude that \eqref{KE equation cyclic} does not hold.  We will leave the proof of Lemma \ref{lemma combinatorial 2} to Section \ref{Sec combinatorial lemma}.

\textbf{Subcase III (b).} $k \geq t_n$.

In this case, $|T|+1\leq |T|+t_j\leq |T|+k$ for all $1\leq j\leq n$. Thus we have, by the Taylor expansion,
\begin{equation*}
Q=\prod_{j=2}^n\binom{n+1+k-t_j}{k-t_j}mx^{k-t_j}+\text{h.o.t}.
\end{equation*}
Note that all other terms in \eqref{KE equation cyclic} have the same Taylor expansions as in the case III (b). As before, in order to disprove \eqref{KE equation cyclic}, it is sufficient to verify the following lemma, whose proof is also delayed to Section \ref{Sec combinatorial lemma}.
\begin{lemma}\label{lemma combinatorial 3}
	Let $2\leq k\leq m-1$ and  $n \geq 2.$ Let $1=t_1\leq \cdots \leq t_n\leq k$. If
	\begin{equation*}
	k=m-\sum_{j=1}^nt_j,
	\end{equation*}
	then
	\begin{align*}
	k\,\binom{n+k}{k}^{n+2}>m\,\binom{n+k+m}{n}	\prod_{j=1}^n\binom{n+1+k-t_j}{k-t_j}.
	\end{align*}
\end{lemma}

\section{Proof of the Combinatorial lemmas}\label{Sec combinatorial lemma}
In this section, we shall prove Lemmas \ref{lemma combinatorial 1},  \ref{lemma combinatorial 2}, and \ref{lemma combinatorial 3}.
\subsection{Proof of Lemma \ref{lemma combinatorial 1}}
For the reader's convenience, we restate Lemma \ref{lemma combinatorial 1} here.
\begin{lemma}
	Suppose $m,n\geq 2$, $1\leq a\leq n$ and $1=t_1=\cdots=t_a<t_{a+1}\leq \cdots\leq t_n\leq m-1$. If $n+2=(m+1)(n-a+1)-\sum_{j=a+1}^n t_j$, then
	\begin{equation}\label{lemma combinatorial 1 eq}
	(n+1)^{n+1}(m+1)>m\binom{n+m+1}{m}\binom{m+n+2-t_{a+1}}{m+1-t_{a+1}}\cdots \binom{m+n+2-t_{n}}{m+1-t_{n}}.
	\end{equation}
\end{lemma}
\begin{proof}
	We divide the proof into two cases.

\textbf{Case I.} $n=2$.

In this case, by the assumption of Lemma \ref{lemma combinatorial 1}, we have $4=(m+1)(3-a)-\sum_{j=a+1}^n t_j$ and $1\leq a\leq 2$.

Suppose $a=1$. Then $2(m+1)=4+t_2\leq m+3$, which yields $m\leq 1$. This contradicts the assumption $m\geq 2$. Thus we have $a=2$. It follows that $t_2=1$ and $m=3$. A straightforward computation shows
\begin{align*}
	\text{LHS of \eqref{lemma combinatorial 1 eq}}=108>60=\text{RHS of \eqref{lemma combinatorial 1 eq}}.
\end{align*}
So this case is verified.

\textbf{Case II.} $n\geq 3$.

We first prove the following elementary combinatorial inequality, which will be used in the proof.
\begin{lemma}
	For any integers $n, k\geq 3$, we have
	\begin{equation}\label{elementary inequality in Case I}
		\binom{n+k}{k-1}<(n+1)^{k-1}.
	\end{equation}
\end{lemma}

\begin{proof}
	\begin{align*}
		\binom{n+k}{k-1}\cdot (n+1)^{-(k-1)}=\prod_{t=1}^{k-1}\frac{(n+1+t)}{t\cdot (n+1)}=\frac{(n+2)(n+3)}{2(n+1)^2} \cdot \prod_{t=3}^{k-1}\frac{(n+1+t)}{t\cdot (n+1)}.
	\end{align*}
	Since $n\geq 3$,
	\begin{align*}
		2(n+1)^2-(n+2)(n+3)=n^2-n-4\geq 2>0,
	\end{align*}
which implies that
$$
\frac{(n+2)(n+3)}{2(n+1)^2}<1
$$
	When $t\geq 3$,
	\begin{align*}
		t(n+1)-(n+1+t)=n(t-1)-1\geq 2n-1>0,
	\end{align*}
which implies that
$$
\frac{(n+1+t)}{t\cdot (n+1)}<1.
$$
	The result therefore follows.
\end{proof}

Recall that $1\leq t_j\leq m-1$ for any $1\leq j\leq n$. By applying \eqref{elementary inequality in Case I} with $k=m+2-t_j$, we get
\begin{equation*}
	\binom{m+n+2-t_j}{m+1-t_j}<(n+1)^{(m+1-t_j)}.
\end{equation*}
Thus,
\begin{align*}
	\text{RHS of \eqref{lemma combinatorial 1 eq}}<m(n+1)^{m+\sum_{j=a+1}^n(m+1-t_j)}=m(n+1)^{(n+1)}<\text{LHS of \eqref{lemma combinatorial 1 eq}}.
\end{align*}
So the proof is complete also in Case II.
\end{proof}

\subsection{Proof of Lemma \ref{lemma combinatorial 2} and Lemma \ref{lemma combinatorial 3}}

We will prove a slightly more general result.
\begin{lemma}\label{main lemma}
	Let $k, m, n$ be integers such that $1\leq k\leq m-1$ and $n \geq 2$. Let $\lambda=(\lambda_1,\cdots \lambda_n)\in \mathbb{Z}^n$ satisfy $\lambda_j\leq k$ for each $1\leq j\leq n$. If
	\begin{equation*}
		m-k=\sum_{j=1}^n\lambda_j,
	\end{equation*}
	then
	\begin{align}\label{main lemma eq}
	k\,\binom{n+k}{k}^{n+2}>m\,\binom{n+k+m}{n}	\prod_{j=1}^n\binom{n+1+k-\lambda_j}{k-\lambda_j}.
	\end{align}
\end{lemma}

Clearly, Lemma \ref{lemma combinatorial 2} follows from \ref{main lemma} by taking $(\lambda_1,\cdots,\lambda_n)=(t_1,\cdots,t_a,t_{a+1}-m,\cdots, t_n-m)$. Lemma \ref{lemma combinatorial 3} follows from \ref{main lemma} by taking $(\lambda_1,\cdots,\lambda_n)=(t_1,\cdots,t_n)$.

\begin{proof}[Proof of Lemma \ref{main lemma}]
	We divide the proof into several steps.
	
\textbf{Step 1.} We show that it is actually sufficient to prove \eqref{main lemma eq} for $0\leq \lambda_j\leq k$ for all $1\leq j\leq n$.

We begin this step with the following elementary combinatorial lemma.
\begin{lemma}\label{lemma rearrangement}
	Let $n\in \mathbb{N}$ and let $s,t$ be integers such that $s+1<t\leq k$. Then we have
	\begin{equation}\label{lemma rearrangement eq}
		\binom{n+1+k-s}{k-s}\binom{n+1+k-t}{k-t}< \binom{n+1+k-(s+1)}{k-(s+1)}\binom{n+1+k-(t-1)}{k-(t-1)}.
	\end{equation}
\end{lemma}

\begin{proof}
	A straightforward computation gives
	\begin{align*}
		\frac{\text{LHS of }\eqref{lemma rearrangement eq}}{\text{RHS of }\eqref{lemma rearrangement eq}}
		=\frac{n+1+k-s}{k-s} \cdot \frac{k+1-t}{n+2+k-t}.
	\end{align*}
	Note that
	\begin{align*}
		(n+1+k-s)(k+1-t)-(k-s)(n+2+k-t)=(n+1)(s+1-t)< 0.
	\end{align*}
	The result thus follows immediately.
\end{proof}

Now we fix $m,n,k$ and apply Lemma \ref{lemma rearrangement} to the product $\prod_{j=1}^n \binom{n+1+k-\lambda_j}{k-\lambda_j}$ in the right hand side of \eqref{main lemma eq}. Suppose $\lambda_{j_1}<0$ for some $1\leq j_1\leq n$. Since $\sum_{j=1}^n\lambda_j=m-k\geq 1$, there is some $1\leq j_2\leq n$ such that $\lambda_{j_2}>0$. We change $\binom{n+1+k-\lambda_{j_1}}{k-\lambda_{j_1}}\binom{n+1+k-\lambda_{j_2}}{k-\lambda_{j_2}}$ to $\binom{n+1+k-(\lambda_{j_1}+1)}{k-(\lambda_{j_1}+1)}\binom{n+1+k-(\lambda_{j_2}-1)}{k-(\lambda_{j_2}-1)}$, i.e., use $\lambda_{j_1}+1$ as the new $\lambda_{j_1}$ and use $\lambda_{j_2}-1$ as the new $\lambda_{j_2}$. Then the sum $\sum_{j=1}^n \lambda_j$ is still equal to $m-k$, and the value of the right hand side of \eqref{main lemma eq} becomes larger. We keep doing this if there is some $\lambda_j<0$ for some $1\leq j \leq n$. Then we finally get $0\leq \lambda_j\leq k$ for all $1\leq j\leq n$ and $m-k=\sum_{j=1}^n\lambda_j$ still holds; and the process will not make the value of the right hand side of \eqref{main lemma eq} smaller. So we only need to prove \eqref{main lemma eq} with the additional condition $\lambda_j\geq 0$ for all $1\leq j \leq n$.

From now on, we will assume $\lambda_j\geq 0$ for all $1\leq j\leq n$. As $\sum_{j=1}^n\lambda_j=m-k\geq 1$, without loss of generality, we can further assume $\lambda_1\geq 1$.

\textbf{Step 2.} We show that it is actually sufficient to prove \eqref{main lemma eq} for $\lambda_1=1$ and $\lambda_2=\cdots=\lambda_n=0$.

For simplicity, we denote the right hand side of \eqref{main lemma eq} by $F$:
\begin{equation}\label{F defn eq}
F(n,k,\lambda):=m\binom{n+k+m}{n}	\prod_{j=1}^n\binom{n+1+k-\lambda_j}{k-\lambda_j},
\end{equation}
where $m=k+\sum_{j=1}^n\lambda_j$ and $\lambda=(\lambda_1,\cdots,\lambda_n)$.
The function $F$ has the following property.
\begin{lemma}\label{lemma F}
	Suppose $n, k \geq 1$ and $\lambda=(\lambda_1,\cdots,\lambda_n) \in \mathbb{Z}^n$ with $0\leq \lambda_j\leq k$ for each $1\leq j\leq n$. If $\lambda_{j_1}\geq 1$ for some $1\leq j_1\leq n$, then
	\begin{equation}\label{lemma F eq}
		F(n,k,\lambda)\leq F(n,k,\lambda-e_{j_1}),
	\end{equation}
	where $e_{j_1}=(0,\cdots,0,1,0\cdots,0)$ is the unit vector along the $j_1$-th direction in $\bR^n$.
	
	Consequently, if $\lambda_1\geq 1$, and all other  $\lambda_j'$s are nonnegative, then
	\begin{equation*}
		F(n,k,\lambda)\leq F(n,k,e_1).
	\end{equation*}
\end{lemma}

\begin{proof}
	We cancel the common combinatorial factors in \eqref{lemma F eq}, and write it as
	\begin{equation*}
	m \binom{n+k+m}{n} \binom{n+1+k-\lambda_{j_1}}{k-\lambda_{j_1}}\leq(m-1) \binom{n+k+m-1}{n}\binom{n+2+k-\lambda_{j_1}}{k-\lambda_{j_1}+1},
	\end{equation*}
	where $m=k+\sum_{j=1}^n\lambda_j$.
	
	By expanding the remaining combinatorial terms and further canceling common factors, we deduce that \eqref{lemma F eq} is equivalent to
	\begin{equation*}
	m\cdot \frac{n+k+m}{k+m}\leq (m-1)\cdot\frac{n+2+k-\lambda_{j_1}}{k-\lambda_{j_1}+1}.
	\end{equation*}
	Clearly, the right hand side is increasing with respect to $1\leq \lambda_{j_1}\leq k$. Thus, it is sufficient to prove
	\begin{equation}\label{lemma F eq simplified}
	m\cdot \frac{n+k+m}{k+m}\leq (m-1)\cdot\frac{n+k+1}{k},
	\end{equation}	
	A straightforward computation shows that
	\begin{align*}
	\eqref{lemma F eq simplified} &\iff m\cdot \frac{n}{k+m}+1\leq (m-1)\cdot \frac{n+1}{k}
	\\
	&\iff mnk+k^2+mk\leq mnk+mk-nk-k+(m^2-m)(n+1)
	\\
	&\iff k^2+nk+k\leq (m^2-m)(n+1).
	\end{align*}
	The last inequality follows immediately by the fact $m=k+\sum_{j=1}^n\lambda_j\geq k+1$. So the proof is finished.
	
\end{proof}

Thus, with the help of Lemma \ref{lemma F}, it suffices to prove \eqref{main lemma eq} for $\lambda_1=1$, $\lambda_2=\cdots=\lambda_n=0$ and $m=k+1$. That is, we only need to show that
\begin{equation} \label{main lemma simplified eq}
k\,\binom{n+k}{k}^{n+2}>(k+1)\,\binom{n+2k+1}{n}\binom{n+k}{k-1}	\binom{n+k+1}{k}^{n-1}.
\end{equation}

\textbf{Step 3.} We complete the proof of Lemma \ref{main lemma}, by proving \eqref{main lemma simplified eq} for any $n\geq 2$, $k \geq 1$.

Let us further simplify \eqref{main lemma simplified eq} to the following equivalent inequalities.
\begin{align*}
\eqref{main lemma simplified eq} &\iff k\cdot \frac{(n+k)!^{2}}{n!^{2}k!^{2}}>(k+1)\cdot\frac{(n+2k+1)!}{n!(2k+1)!}\cdot \frac{k}{n+1}\cdot \frac{(n+k+1)^{n-1}}{(n+1)^{n-1}}
\\
&\iff
\frac{(n+1)^n}{(k+1)(n+k+1)^{n-1}}\cdot \frac{(n+k)!^2(2k+1)!}{n!k!^2(n+2k+1)!}>1.
\end{align*}

Denote the left hand side term in the last inequality by $L(n,k)$, i.e.,
\begin{equation*}
	L(n,k):=\frac{(n+1)^n}{(k+1)(n+k+1)^{n-1}}\cdot \frac{(n+k)!^2(2k+1)!}{n!k!^2(n+2k+1)!}.
\end{equation*}
It remains to prove $L(n,k)>1$ for $n\geq 2$, $k \geq 1$.

\begin{lemma}\label{lemma L increasing}
	Given nonnegative integers $n,k$, we have
	\begin{equation}
		L(n,k)\leq L(n+1,k).
	\end{equation}
\end{lemma}

\begin{proof}
	Set $Q(n,k):=L(n+1,k)/L(n,k)$. Then
	\begin{equation*}
		Q(n,k)=\frac{(n+2)^{n+1}}{(n+1)^{n+1}} \cdot\frac{(n+k+1)^{n+1}}{(n+k+2)^n(n+2k+2)}.
	\end{equation*}
	Regarding $k$ as a real variable in $[0,\infty)$, we take the logarithmic derivative of $Q(n,k)$ with respect to $k$:
	\begin{align*}
		\frac{\partial\log Q(n,k)}{\partial k}=&\frac{n+1}{n+k+1}-\frac{n}{n+k+2}-\frac{2}{n+2k+2}
		\\
		=&\frac{2n+k+2}{(n+k+1)(n+k+2)}-\frac{2}{n+2k+2}
		\\
		=&\frac{nk}{(n+k+1)(n+k+2)(n+2k+2)}.
	\end{align*}
	It follows that for given $n\geq 0$, $Q(n,k)$ is increasing with respect to $k\geq 0$. Thus, for $n,k\geq 0$, we have
	\begin{equation*}
		Q(n,k)\geq Q(n,0)=1.
	\end{equation*}
	That is the desired result.
\end{proof}

Now, for $n \geq 2, k \geq 1$, Lemma \ref{lemma L increasing} yields that
\begin{align*}
	L(n,k)\geq L(2,k)=&\frac{3^2}{(k+1)(k+3)}\cdot \frac{(k+2)!^2(2k+1)!}{2! k!^2 (2k+3)!}
	\\
	=&\frac{9(k+1)^2(k+2)^2}{2(k+1)(k+3)(2k+2)(2k+3)}
	\\
	=&\frac{9(k^2+4k+4)}{4(2k^2+9k+9)}>1.
\end{align*}
The proof of Lemma \ref{main lemma} is now complete.
\end{proof}

\bibliographystyle{plain}
\bibliography{references}

\end{document}